\documentclass[11pt, reqno]{amsart}

\usepackage{amssymb,amsmath,amsfonts,amsthm,mathrsfs}

\usepackage{longtable}
\usepackage{color}
\usepackage[dvipsnames,table]{xcolor}
\usepackage{graphicx}
\usepackage{hyperref}
\usepackage{multicol}

\usepackage{fancyvrb}
\usepackage[toc,page]{appendix}

\usepackage{tikz}
\usepackage{amsmath}
\usepackage[english]{babel}
\usepackage{amsthm}
\usetikzlibrary{automata,arrows,positioning,calc}
\usepackage{amssymb}
\usepackage{bm}
\usepackage{subcaption}
\usepackage{comment}

\usepackage{listings}
\definecolor{codegreen}{rgb}{0,0.6,0}
\definecolor{codegray}{rgb}{0.5,0.5,0.5}
\definecolor{codepurple}{rgb}{0.58,0,0.82}
\definecolor{backcolour}{rgb}{0.95,0.95,0.92}

\usepackage{cite}

  \renewenvironment{thebibliography}[1]{
    \begin{oldthebibliography}{#1}
      \setlength{\parskip}{0ex}
      \setlength{\itemsep}{0ex}
  }
  {
    \end{oldthebibliography}
  }

\lstdefinestyle{mystyle}{
	backgroundcolor=\color{backcolour},   
	commentstyle=\color{codegreen},
	keywordstyle=\color{magenta},
	numberstyle=\tiny\color{codegray},
	stringstyle=\color{codepurple},
	basicstyle=\ttfamily\footnotesize,
	breakatwhitespace=false,         
	breaklines=true,                 
	captionpos=b,                    
	keepspaces=true,                 
	numbers=left,                    
	numbersep=5pt,                  
	showspaces=false,                
	showstringspaces=false,
	showtabs=false,                  
	tabsize=2
}

\theoremstyle{plain}
\newtheorem{thm}{Theorem}[section]
\newtheorem{prop}[thm]{Proposition}

\newtheorem{lemma}[thm]{Lemma}

\theoremstyle{definition}
\newtheorem{defin}[thm]{Definition}

\newtheorem{rmk}[thm]{Remark}
\newtheorem{conj}[thm]{Conjecture}
\newtheorem{question}[thm]{Question}

\usepackage{textcomp}
\usepackage{float}

\usepackage{array}
\usepackage{graphicx}

\graphicspath{ {./images/} }

\calclayout 

\begin{document}

	\title{A note on non-regular Bonnet-Myers Sharp Graphs }

    \author[Cushing]{David Cushing}
    \address{Department of Mathematics, University of Manchester, Manchester, Great Britain}
    \email{david.cushing@manchester.ac.uk}
    \author[Stone]{Adam J. Stone}
    \address{Department of Mathematical Sciences, Durham University, Durham, Great Britain}
    \email{adam.stone2@durham.ac.uk}

	\date{\today}
	
	\begin{abstract}
Self-centred regular graphs which are Ollivier-Ricci Bonnet-Myers sharp have been completely classified. When the conditions of self-centeredness and regularity are removed it is an open problem on what the classification is. We present a complete classification of Bonnet-Myers sharp graphs with a diameter 2 and show that there exists Bonnet-Myers sharp graphs of diameter 3, 4 and 6 that belong to a family of graphs called symmetrical antitees.
	\end{abstract}
	
	\maketitle
	
	%\tableofcontents
	
 \section{Introduction}
    A fundamental question in geometry is in which way local properties determine the global structure of a space. A famous result of this kind is the Bonnet-Myers Theorem \cite{My41} for complete $n$-dimensional Riemannian manifolds $M$ with $K = \inf {\rm Ric}_M(v) > 0$, where the infimum is taken over all unit tangent vectors $v$ of $M$. Under this condition, $M$ is compact and its diameter satisfies
\begin{equation} \label{eq:BM_RG_ineq}
{\rm diam}(M) \le \pi \sqrt{\frac{n-1}{K}}. 
\end{equation}
Moreover, Cheng's Rigidity Theorem \cite{Cheng75} states that this diameter estimate (\ref{eq:BM_RG_ineq}) has equality if and only if $M$ is the $n$-dimensional round sphere.
\\
In the setting of graphs, there is a discrete analogue of Bonnet-Myers theorem for Ollivier-Ricci curvature (see, e.g., \cite{Ollivier, LLY}) which states that for a graph $G$ with positive curvature lower bound $K = \inf_{x\sim y} \kappa_{LLY}(x, y) > 0$, its diameter satisfies an upper bound
\begin{equation} \label{eq:BM_Intro}
{\rm diam}(G) \le \frac{2}{K}. 
\end{equation}
Here $\kappa_{LLY}$ is the Ollivier-Ricci curvature, introduced in \cite{LLY} and modified from \cite{Ollivier}. This notion is introduced in Section $2.2.$ A natural question is which graphs satisfy inequality $(2)$ with equality, such graphs are called {\it Bonnet-Myers sharp.}
\\
In \cite{CushR} the authors classify all Bonnet-Myers sharp graphs under the assumptions of regularity and self-centeredness:

\begin{thm}[see \cite{CushR}]\label{thm:MainRigidityThm}
  Self-centered regular Bonnet-Myers sharp graphs are precisely the following
  graphs:
  \begin{enumerate}
  \item hypercubes $Q^n$, $n \ge 1$;
  \item cocktail party graphs $CP(n)$, $n \ge 3$; 
  \item the Johnson graphs $J(2n,n)$, $n \ge 3$;
  \item even-dimensional demi-cubes $Q^{2n}_{(2)}$, $n \ge 3$; 
  \item the Gosset graph;
  \end{enumerate}
  and Cartesian products $G = G_{1}\times \cdots\times G_k$ where each $G_i$ are one of (1)-(5) satisfying the condition
  $$\frac{{\rm deg}(G_1)}{{\rm diam}(G_1)} = \ldots = \frac{{\rm deg}(G_k)}{{\rm diam}(G_k)}.$$
\end{thm}
See section $2.1$ for the relevant graph theoretical notions.

The authors also showed that the self-centeredness assumption is not necessary for regular Bonnet-Myers sharp graphs of diameter 2:
\begin{thm}[see \cite{CushR}]\label{thm:CPRigidityThm}
Let $G$ be a regular graph of diameter $2.$ The following are equivalent
  \begin{enumerate}
  \item $G$ is Bonnet-Myers sharp.
  \item $G$ is isomorphic to a cocktail party graphs $CP(n)$ for some $n \ge 2.$ 
  \end{enumerate}
\end{thm}
    
    When the regularity condition in the above theorem is dropped, a larger set of graphs are Bonnet-Myers sharp. Thus, the aim of this paper is to extend this classification result to the case of not necessarily regular diameter 2 graphs. 
    \\
    Our main result is the following classification of diameter 2 Bonnet-Myers sharp graphs:
    \begin{thm}\label{thm:2RigidityThm}
    Let G be a graph with diameter 2 on $n$ vertices. The following are equivalent:
    \begin{enumerate}
        \item 
        $G$ is Bonnet-Myers sharp.
        \item
        $G$ can be obtained by deleting a matching from the complete graph on $n$ vertices $K_n.$
    \end{enumerate}
    \end{thm}
    The proof of Theorem \ref{thm:2RigidityThm} follows from Proposition \ref{prop:IsBMSharp} and Theorem \ref{thm:classification} proved below.

    We also introduce a family of symmetrical antitrees and prove the following theorem:
    \begin{thm}
         A Symmetrical antitree $G=\mathcal{AT}((a_k))$ is Bonnet-Myers Sharp if and only if all radial edges have curvature equal to $2/diam(G)$.
    \end{thm}
    This theorem follows from the antitree equations formulated in \cite{cushing2018curvature} and is used to show that
    \begin{lemma}
        There are an infinite number of symmetrical antitrees of diameter 4 or 6 that are Bonnet-Myers Sharp.
    \end{lemma}
	
    \section{Definitions}
    \subsection{Graph theoretical notation}
    Throughout this paper, we restrict our graphs $G=(V,E)$ to be undirected, simple, unweighted, finite, and connected, where $V$ is the vertex set and $E$ is the edge set of $G$.
    
    We write $x \sim y$ if there exists an edge between the vertices $x$ and $y$ and we write $x \not\sim y$ if there does not exist an edge between the vertices $x$ and $y$. The degree of a vertex $x \in V$ is denoted by $d_x$.

    For any two vertices $x,y\in V$, the (combinatorial) distance $d(x,y)$ is the length (i.e. the number of edges) in a shortest path from $x$ to $y$. 
    
    The {\it diameter} of $G$ is denoted by ${\rm diam}(G)= \max_{x,y\in V} d(x,y)$. 
    
    A vertex $x\in V$ is called a {\it pole} if there exists a vertex $y\in V$ such that $d(x,y)={\rm diam}(G)$, in which case $y$ will be called an {\it antipole} of $x$ (with respect to $G$). A graph $G$ is called {\it self-centered} if every vertex is a pole.
    
    A family of graphs that we will refer to are the {\it cocktail party graphs} $CP(n)$ obtained by removal of a perfect
    matching from the complete graph $K_{2n}.$

    \subsection{ Ollivier-Ricci curvature }
Ollivier-Ricci curvature was introduced in \cite{Ollivier} and is based on optimal transport. A fundamental concept in optimal transport is the
Wasserstein distance between probability measures.
\begin{defin}
Let $G = (V,E)$ be a graph. Let $\mu_{1},\mu_{2}\in V\mapsto[0,1]$ be two probability measures on $V$. The {\it Wasserstein distance} $W_1(\mu_{1},\mu_{2})$ between $\mu_{1}$ and $\mu_{2}$ is defined as
\begin{equation} \label{eq:W1def}
W_1(\mu_{1},\mu_{2})=\inf_{\pi} \sum_{y\in V}\sum_{x\in V} d(x,y)\pi(x,y),
\end{equation}
where the infimum runs over all transportation plans $\pi:V\times  V\rightarrow [0,1]$ satisfying
$$\mu_{1}(x)=\sum_{y\in V}\pi(x,y),\:\:\:\mu_{2}(y)=\sum_{x\in V}\pi(x,y).$$
\end{defin}
The transportation plan $\pi$ moves a mass
distribution given by $\mu_1$ into a mass distribution given by
$\mu_2$, and $W_1(\mu_1,\mu_2)$ is a measure for the minimal effort
which is required for such a transition.
\\
\\
If $\pi$ attains the infimum in \eqref{eq:W1def} we call it an {\it
  optimal transport plan} transporting $\mu_{1}$ to $\mu_{2}$.
\\
\\
We define the following probability distributions $\mu_x$ for any
$x\in V,\: p\in[0,1]$:
$$\mu_x^p(z)=\begin{cases}p,&\text{if $z = x$,}\\
\frac{1-p}{d_x},&\text{if $z\sim x$,}\\
0,& \mbox{otherwise.}\end{cases}$$

These probability distributions can be thought of as a simple random work with idleness parameter $p$.

\begin{defin}
The $ p-$Ollivier-Ricci curvature on an edge $x\sim y$ in $G=(V,E)$ is
$$\kappa_{ p}(x,y)=1-W_1(\mu^{ p}_x,\mu^{ p}_y),$$
where $p$ is called the {\it idleness}.

The Ollivier-Ricci curvature introduced by Lin-Lu-Yau in
\cite{LLY}, is defined as
$$\kappa_{LLY}(x,y) = \lim_{ p\rightarrow 1}\frac{\kappa_{ p}(x,y)}{1- p}.$$
\end{defin}
More details on the motivation of this notion can be found in \cite{Ollivier}.

A fundamental concept in the optimal transport theory is Kantorovich duality. First we recall the notion
of a 1--Lipschitz functions and then state Kantorovich duality.

\begin{defin}
Let $G=(V,E)$ be a locally finite graph, $\phi:V\rightarrow\mathbb{R}.$ We say that $\phi$ is $1$-Lipschitz if 
$$|\phi(x) - \phi(y)| \leq d(x,y)$$
for all $x,y\in V.$ Let \textrm{1--Lip} denote the set of all $1$--Lipschitz functions. 
\end{defin}

\begin{thm}[Kantorovich duality]\label{Kantorovich}
Let $G = (V,E)$ be a locally finite graph. Let $\mu_{1},\mu_{2}$ be two probability measures on $V$. Then
$$W_1(\mu_{1},\mu_{2})=\sup_{\substack{\phi:V\rightarrow \mathbb{R}\\ \phi\in \textrm{\rm{1}--{\rm Lip}}}}  \sum_{x\in V}\phi(x)(\mu_{1}(x)-\mu_{2}(x)).$$
\\
\\
If $\phi \in \textrm{1--Lip}$ attains the supremum we call it an \emph{optimal Kantorovich potential} transporting $\mu_{1}$ to $\mu_{2}$.
\end{thm}

The following result on some properties of $p \mapsto \kappa_p(x,y)$
for $x \sim y$ allows us to choose a convenient choice of idleness parameter $p$ for calculation considerations and then calculate $\kappa_{LLY}(x,y)$  from $\kappa_{p}(x,y) $.

\begin{thm}[see \cite{BCLMP}]\label{thm:idleness}
  Let $G=(V,E)$ be a locally finite graph. Let $x,y\in V$ with
  $x\sim y.$ Then the function $p \mapsto \kappa_{p}(x,y)$ is concave
  and piecewise linear over $[0,1]$ with at most $3$ linear
  parts. Furthermore $\kappa_{p}(x,y)$ is linear on the intervals
  \begin{equation*}
    \left[0,\frac{1}{{\rm{lcm}}(d_{x},d_{y})+1}\right]\:\:\: {\rm and} \:\:\:\left[\frac{1}{\max(d_{x},d_{y})+1},1\right].
  \end{equation*}
  Thus, if $p\in \left[\frac{1}{\max(d_{x},d_{y})+1},1\right],$ then 
  $$\kappa_{LLY}(x,y) = \frac{1}{1-p}\kappa_{p}(x,y).$$
\end{thm}

Let us now state the discrete Bonnet-Myers Theorem for Ollivier-Ricci
curvature and introduce the associated notion of Bonnet-Myers sharpness
for this curvature notion:

\begin{thm}[Discrete Bonnet-Myers, see \cite{Ollivier,LLY}]
  \label{thm:DBM}
  Let $G= (V,E)$ be a connected graph and $\inf_{x \sim y}
  \kappa_{LLY}(x,y) > 0$. Then $G$ has finite diameter $L = {\rm diam}(G) <
  \infty$ and  
 \begin{equation} \label{eq:BM_est} 
  \inf_{x \sim y} \kappa_{LLY}(x,y) \le \frac{2}{L}. 
  \end{equation}
  We say that such a graph $G$ is
  {\bf{\textit{Bonnet-Myers sharp}}} (with respect
    to Ollivier-Ricci curvature) if \eqref{eq:BM_est} holds with
    equality.
\end{thm}
    
    \section{Bonnet-Myers sharp graphs of diameter 2}
    We now introduce the diameter 2 graphs which are Bonnet-Myers sharp. 
    \begin{defin}
    Let $a,b$ be non-negative integers with $a > 0$. Let $H(a,b)$ be a graph on $2a+b$ vertices $u_1,\ldots u_{2a}, v_1, \ldots v_b$. For every $1\leq i \leq a$ we set $u_{2i-1}\sim u_{2i}$ and there are no other pairs of adjacent vertices.
    \\
    Let $G(a,b)$ be the graph complement of $H(a,b)$. Note that $G(a,b)$ has diameter $2$ unless $(a,b) = (1,0)$ in which case $G(a,b)$ is the graph of two isolated vertices.
    \end{defin}
    
    \begin{rmk}\label{rmk:Gab}
    If $b = 0$ then $G(a,0)$ is the cocktail party graph on $2a$ vertices and is thus Bonnet-Myers sharp for $a\geq 2$. Otherwise when $b\neq 0$ the graph $G(a,b)$ is not regular.
    \\
    Note that a graph $G$ is isomorphic to $G(a,b)$ for some $a,b$ if and only if $G$ can be obtained by deleting some matching from the complete graph on $n$ vertices, $K_n.$
    \end{rmk}
    
    \begin{prop}\label{prop:IsBMSharp}
    Let $a,b$ be positive integers and set $G = G(a, b)$ with usual vertex set $u_1,\ldots u_{2a}, v_1, \ldots v_b$. Then
    \begin{enumerate}
        \item 
        If $a \geq 2,$ then for all $1\leq i < j \leq 2a$ such that $u_i\sim u_j,$ we have
        $$\kappa_{LLY}(u_i,u_j) = 1.$$
        \item 
        For all $1\leq i \leq 2a, 1\leq j \leq b,$ we have
        $$\kappa_{LLY}(u_i,v_j) = 1.$$
        \item 
        If $b \geq 2$ then for all $1\leq i < j \leq b,$ we have
        $$\kappa_{LLY}(v_i,v_j) = \frac{2a+b}{2a+b-1} > 1.$$
    \end{enumerate}
    Therefore
    $$\inf_{(x,y)\in E(G)}\kappa_{LLY} (x,y) = 1,$$
    and thus $G$ is Bonnet-Myers sharp.
    \end{prop}
    \begin{proof}
    \begin{enumerate}
    \item 
We work with idleness $p = \frac{1}{2a+b-1}.$ Without loss of generality we may assume that $i=1$ and $j=3.$ Note that $d_{u_1} = d_{u_3} = 2a+b-2.$
Observe that
$$\mu^p_{u_1}(u_4) > \mu^p_{u_3}(u_4), \:\:\: \mu^p_{u_1}(u_2) < \mu^p_{u_3}(u_2),  $$
and 
$$\mu^p_{u_1}(x) = \mu^p_{u_3}(x)$$
for all $x\neq u_2,u_4.$
        
Thus when transporting $\mu^p_{u_1}$ to $\mu^p_{u_3}$ the only vertex that gains mass is $u_2.$ Note further that this mass can be transported over a distance of 1. Thus
$$W_{1}(\mu^p_{u_1}, \mu^p_{u_3}) \leq \mu^p_{u_3}(u_2) - \mu^p_{u_1}(u_2) =\frac{1}{2a+b-1}. $$
We verify that this is in fact equality by constructing the following $\phi\in 1-$Lip,
$$
\phi(x) = \begin{cases}
			1, &  x = u_4\\
            0, & \text{otherwise}
		 \end{cases}.
$$
        
Then, by Theorem \ref{Kantorovich},
$$W_{1}(\mu^p_{u_1}, \mu^p_{u_3})\geq  \sum_{x}\phi(x)(\mu^{p}_{u_1}(x)-\mu^{p}_{u_3}(x)) = \frac{1}{2a+b-1}.$$
Therefore 
$$\kappa_{LLY}(u_1,u_3) = \frac{1}{1-p}\kappa_{p}(u_1, u_3) = \frac{2a+b-1}{2a+b-2} (1 - W_{1}(\mu^p_{u_1}, \mu^p_{u_3})) = 1. $$
\item
We work with idleness $p = \frac{1}{2a+b}.$ Without loss of generality we may assume that $i=j=1.$ Note that $d_{u_1} = 2a+b-2, d_{v_1} = 2a+b-1.$

Observe that
$$\mu^p_{u_1}(u_2) < \mu^p_{v_1}(u_2), \:\:\: \mu^p_{u_1}(u_1) = \mu^p_{v_1}(u_1), $$
and 
$$\mu^p_{u_1}(x) \geq \mu^p_{v_1}(x)$$
for all $x\neq u_2.$

Thus when transporting $\mu^p_{u_1}$ to $\mu^p_{v_1}$ the only vertex that gains mass is $u_2.$ Since the only vertex that is distance more than one away from $u_2$ is $u_1$ this mass can be transported over a distance of 1. Thus
$$W_{1}(\mu^p_{u_1}, \mu^p_{v_1}) \leq \mu^p_{v_1}(u_2) - \mu^p_{u_1}(u_2) =\frac{1}{2a+b}. $$
We verify that this is in fact equality by constructing the following $\phi\in 1-$Lip,
$$
\phi(x) = \begin{cases}
			-1, &  x = u_2\\
            0, & \text{otherwise}
		 \end{cases}.
$$

Then, by Theorem \ref{Kantorovich},
$$W_{1}(\mu^p_{u_1}, \mu^p_{v_1})\geq  \sum_{x}\phi(x)(\mu^{p}_{u_1}(x)-\mu^{p}_{v_1}(x)) = \frac{1}{2a+b}.$$
Therefore 
$$\kappa_{LLY}(u_1,v_1) = \frac{1}{1-p}\kappa_{p}(u_1, v_1) = \frac{2a+b}{2a+b-1} (1 - W_{1}(\mu^p_{u_1}, \mu^p_{v_1})) = 1. $$
\item
We work with idleness $p = \frac{1}{2a+b}.$ Note that $\mu^p_{v_i} \equiv \mu^{p}_{v_j}.$ Thus 
$$\kappa_{LLY}(v_i,v_j) = \frac{1}{1-p}\kappa_{p}(v_i,v_j) = \frac{2a+b}{2a+b-1} (1 - W_{1}(\mu^p_{v_i}, \mu^p_{v_j})) = \frac{2a+b}{2a+b-1}. $$
    \end{enumerate}
    
    \end{proof}
    We now show that these graphs are in fact the complete set of Bonnet-Myers sharp graphs of diameter 2. We first prove some properties of these graphs.
    
    \begin{lemma}\label{lemma:MassMoved}
    Let $G = (V,E)$ be a Bonnet-Myers sharp graph of diameter 2. Let $x,y\in V$ with $x\sim y$ and $p\in \left[\frac{1}{\max(d_{x},d_{y})+1},1\right].$ Then
    $$W_{1}\left(\mu^p_{x}, \mu^p_{y}\right) \leq p.$$
    \end{lemma}
    \begin{proof}
    Since $G$ is Bonnet-Myers sharp with diameter 2 we have $\kappa_{LLY}(x,y) \geq 1.$ Applying the definition of $\kappa_p$ and Theorem \ref{thm:idleness} gives 
    $$W_{1}\left(\mu^p_{x}, \mu^p_{y}\right) = 1 - \kappa_{p}(x,y) = 1 - (1-p)\kappa_{LLY}(x,y)\leq p.$$
    \end{proof}
    
    \begin{lemma}\label{lemma:SameDegree}
    Let $G = (V,E)$ be a Bonnet-Myers sharp graph of diameter 2. Let $x,y\in V$ with $x\sim y$ such that $d_x=d_y.$ Then there exists at most one element $x'\in V\setminus\{x,y\}$ such that $x\sim x'$ and $y\not\sim x'.$
    Furthermore, if such an $x'$ exists then there exists a $y'\in V\setminus\{x,y\}$ such that $y\sim y', x\not\sim y'$ and $x'\sim y'.$ 
    \end{lemma}
    In other words all neighbours of $x$ are also neighbours of $y$ except for at most one vertex $x'$ which forms a square as illustrated below.
    \\
\begin{center}
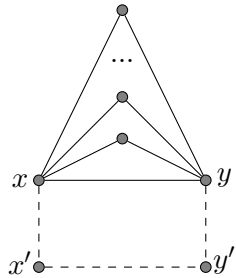

    \begin{tikzpicture}[
        roundnode/.style={circle, draw, fill=black!50,inner sep=0pt, minimum width=4pt},
        ]
        \node[roundnode,label={[xshift=-0.25cm,yshift=-0.3cm] $x$}]      (A)                  {};
        \node[roundnode, label={[yshift=0.25cm] ...}]      (C)                  [above right =of A]{};
        \node[roundnode,]      (F)                  [above =of C]{};
        \node[roundnode,label={[xshift=0.25cm,yshift=-0.3cm] $y$}]      (B)                  [below right =of C]{};
        \node[roundnode,label={[xshift=-0.25cm,yshift=-0.3cm] $x'$}]      (D)                  [below =of A]{};
        \node[roundnode,label={[xshift=0.25cm,yshift=-0.3cm] $y'$}]      (E)                  [below =of B]{};
        \node[roundnode,]      (G)                  [below =of C,yshift=0.6cm]{};
        
        \draw(A)     edge    (B);
        \draw(A)     edge    (C);
        \draw(B)     edge    (C);
        \draw(F)     edge    (A);
        \draw(B)     edge    (F);
        \draw(G)     edge    (A);
        \draw(B)     edge    (G);
        \draw[dashed](A)     edge    (D);
        \draw[dashed](D)     edge    (E);
        \draw[dashed](B)     edge    (E);
    \end{tikzpicture}
    \captionof{figure}{A visual representation of Lemma \ref{lemma:SameDegree}. }
\end{center}

    \begin{proof}
    When transporting $\mu^{\frac{1}{d_x+1}}_{x}$ to $\mu^{\frac{1}{d_x+1}}_{y}$ each neighbour of $x$ which is not a neighbour of $y$ contributes at least $\frac{1}{d_x+1}$ to the cost of $W_{1}\left(\mu^{\frac{1}{d_x+1}}_{x}, \mu^{\frac{1}{d_x+1}}_{y}\right).$
    
    However, by Lemma \ref{lemma:MassMoved}, $W_{1}\left(\mu^{\frac{1}{d_x+1}}_{x}, \mu^{\frac{1}{d_x+1}}_{y}\right)\leq \frac{1}{d_x+1},$ which shows that at most one element satisfies the desired properties of $x'.$ Furthermore, the mass of any such element must travel a distance of at most $1,$ thus showing the existence of $y'.$
    \end{proof}
    
    \begin{lemma}\label{lemma:DifferentDegree}
    Let $G = (V,E)$ be a Bonnet-Myers sharp graph of diameter 2. Let $x,y\in V$ with $x\sim y$ such that $d_x<d_y.$ Then there does not exist any $z\in V\setminus\{x,y\}$ such that $x\sim z$ and $y\not\sim z.$
    \end{lemma}
    \begin{proof}
    Suppose such a $z$ exists. Let $p = \frac{1}{d_x+1}.$ We have
    $$\mu^p_x(z) = \frac{1}{d_x+1}> \:\:\: \mu^p_y(z) = 0,$$
    $$\mu^p_x(x) = \frac{1}{d_x+1}> \:\:\: \mu^p_y(x) = \frac{d_x}{d_y}\frac{1}{d_x+1}.$$
    
    Thus mass must be moved from $x$ and $z$ when transporting $\mu^p_x$ to $\mu^p_y$ and thus
    $$W_{1}(\mu^{p}_{x}, \mu^{p}_{y})\geq \frac{1}{d_x+1}+\frac{d_y - d_x}{d_y}\frac{1}{d_x+1} > \frac{1}{d_x+1} = p,$$
    contradicting Lemma \ref{lemma:MassMoved}.
    \end{proof}

    We are now ready to prove the classification.

    \begin{thm}\label{thm:classification}
    Let $G = (V,E)$ be a Bonnet-Myers sharp graph of diameter 2. Then there exists $a,b\in\mathbb{N}$ with $a>0$ such that $G$ is isomorphic to $G(a,b).$
    \end{thm}
    
    \begin{proof}
    If $G$ is regular then by Theorem \ref{thm:CPRigidityThm} it is isomorphic to a Cocktail Party graph and therefore has the desired form by Remark \ref{rmk:Gab}.
    
    Therefore suppose $x,y\in V$ such that $d_x < d_y$ and let $p = \frac{1}{d_y+1}.$ Since $d_y > d_x,$ there exists a $z\sim y$ such that $d(x,z)=2.$ Note that 
    $$\mu^p_x(z) = 0, \:\:\: \mu^p_y(z) = \frac{1}{d_y+1},$$
    and so $ \frac{1}{d_y+1}$ mass must be moved when transporting $\mu^p_x$ to $\mu^p_y$ and then by applying Lemma \ref{lemma:MassMoved} we see that there exists no vertex $w\neq z$ such that $y\sim w$ and $x\not\sim w.$ Therefore $d_y = d_x+1.$ Furthermore all mass must be transported over a distance of $1$ meaning $u\sim z$ for every $u\sim x.$ Thus $d_z \geq d_x.$ Note this shows that all neighbouring vertices differ in degree by at most $1$.
    
    Since $y$ has a neighbor which is not a neighbour of $z$ (namely $x$) and applying Lemma \ref{lemma:DifferentDegree} we conclude that $d_z \leq d_y.$ Thus $d_z = d_x$ or $d_x+1.$ 
    
    First we suppose $d_z = d_x+1$ and search for a contradiction. We have $d_z = d_y$ and since $z\sim y$ and $z\not\sim x$ there must exist a $s\sim z$ with $s\sim x$ and $s\not\sim y$ by Lemma \ref{lemma:SameDegree}. However by \ref{lemma:DifferentDegree}, all neighbours of $x$ are also neighbours of $y$. Therefore contradicting $s\not\sim y$ and therefore $d_z = d_x.$
    
    We now show that $z$ is the unique vertex distance $2$ from $x$. Suppose that this is not the case and there exists $v\in V$ with $d(x,v) = 2.$ Let $u\in V$ with $x\sim u$ and $u\sim v.$ Since every neighbor of $x$ is a neighbour of $z$ we have $u\sim z.$ Thus $u\sim v, u\sim z, x\not\sim v$ and $x\not\sim z$ which is a contradiction by either Lemma \ref{lemma:SameDegree} or Lemma \ref{lemma:DifferentDegree}.  
    
    Thus $z$ is the unique vertex which is distance $2$ from $x.$ Combining this with $G$ having diameter $2$ shows that $|V| = d_x+2.$ Therefore $d_y = |V| - 1$ showing that $y$ is a neighbour of every other vertex. 

    Therefore every vertex has degree $|V|-2$ or $|V|-1.$ Furthermore every vertex of degree $|V|-2$ has a unique vertex of distance $2$ away from itself and this vertex also has degree $|V|-2$. If for every such vertex of degree $|V|-2$ we add in an edge connecting it to the unique vertex which is distance $2$ away then we obtain the complete graph on $|V|$ vertices, $K_{|V|}.$ Thus $G$ can be obtained by deleting a matching from $K_{|V|},$ which by Remark \ref{rmk:Gab} completes the proof. 
    \end{proof}
    
    \section{Non-regular Bonnet-Myers sharp graphs of diameter greater than 2}
    In this section we construct a non-regular Bonnet-Myers sharp graph for any given diameter greater than one. We first must recall the definition of antitrees.
    \subsection{Antitrees}
    For any given finite or infinite sequence $(a_k)_{1 \le k \le N}$, $N \in \mathbb{N} \cup \{\infty\}$, the {\bf antitree}, denoted by $\mathcal{AT}((a_k)),$ has vertex set $V = \bigcup_{1 \leq k \leq N} V_{k}$ where $|V_k| = a_k.$ The vertices that form $V_1$ is called the \emph{root set}. The edges of $\mathcal{AT}((a_k))$ are defined by 
    \begin{itemize}
    \item any two distinct vertices in the same $V_k$ are connected, 
    \item any root vertex $x \in V_1$ is connected to all vertices in $V_2$, and no vertices in $V_k$, $k \ge 3$,
    \item any vertex $x \in V_k$, $k \ge 2$, is connected to all vertices in $V_{k-1}$ and $V_{k+1}$, and no vertices in $V_l$, $|k-l| \ge 2$.
    \end{itemize}
    Note that the diameter of the finite antitree is equal to $N$. Edges between vertices of the component are referred to as \emph{spherical edges}. Edges between vertices of different components are called \emph{radial edges}. Any radial or spherical edge incident to a vertex in $V_1$ is called \emph{radial} or \emph{spherical root-edges}, respectively. All other edges are called \emph{inner edges}.
    \\
    \begin{thm}[see \cite{cushing2018curvature}]\label{thm:AntitreeCurvatures}
        Let $(a_k)_{1 \le k \le N}$, $N \in \mathbb{N} \cup \{\infty\}$ be a finite sequence with  $a_1=1$, then $G = \mathcal{AT} ((a_k))$ has the following curvatures:
    \begin{itemize}
        \item 
        Radial root edges: If $x \in V_1$ and $y \in V_2$ then
        \begin{equation}\label{RadialRoot}
            \kappa_{LLY}(x,y)=\frac{a_2+1}{a_2+a_3}.
        \end{equation}
        \item 
        Radial inner edges: If $x \in V_k$ and $y \in V_k+1$, $k \geq 2$ then 
        \begin{equation}\label{RadialInner}
            \kappa_{LLY}(x,y)=\frac{2a_k+a_{k+1}-1}{a_{k}+a_{k+1}+a_{k+2}-1}-\frac{2a_{k-1}+a_{k}-1}{a_{k-1}+a_{k}+a_{k+1}-1}.
        \end{equation}
        \item 
        Spherical edges: If $x, y \in V_k$ , $x \neq y$, $k \geq 2$ then
        \begin{equation}\label{Spherical}
            \kappa_{LLY}(x,y)=\frac{a_{k-1}+a_k+a_{k+1}}{a_{k-1}+a_k+a_{k+1}-1}.
        \end{equation}
    \end{itemize} 
    \end{thm}

    \subsection{Symmetrical antitrees}
    $$\begin{array}{c}
a \\
a  \quad a \\
a \quad  b \quad a \\
a \quad b \quad b \quad a \\
a \quad b \quad c \quad b \quad a \\
a \quad b \quad c \quad c \quad b \quad a \\
a \quad b \quad c \quad d \quad c \quad b \quad a \\
a \quad b \quad c \quad d \quad d \quad c \quad b \quad a \\
a \quad b \quad c \quad d \quad e \quad d \quad c \quad b \quad a \\
\end{array}
    $$
    We now introduce a particular family of antitrees which are based on the symmetrical triangle above.
    \begin{defin}
    For an antitree of diameter $L\in\mathbb{N}$, the antitree is symmetrical if, for even $L$ we have $|V_k|=|V_{L+2-k}|$ for all $k \in \{1,..,L/2\}$ and for odd $L$ we have $|V_k|=|V_{L+2-k}|$ for all $k \in \{1,..,(L+1)/2\}$. 
    \end{defin}
    In this paper, we only consider the family of symmetrical antitrees with $|V_1|=1$.
    
    Symmetrical antitrees with $L=1$ are complete graphs and so by Theorem \ref{thm:MainRigidityThm} are not Bonnet-Myers sharp. From Theorem \ref{thm:2RigidityThm}, it is trivial to show that all symmetrical antitrees with diameter 2 and $|V_1|=1$ are Bonnet-Myers Sharp. Furthermore, using the formulas above, one can show that the only symmetric antitree with diameter 3 and $|V_1|=1$ that is Bonnet-Myers sharp graph is $\mathcal{AT}(1,3,3,1)$.

    The idea of the following proofs are to force the curvature between nodes in $V_1$ and nodes in $V_2$ to satisfy Bonnet-Myers theorem with equality by choosing appropriate $|V_2|$ and $|V_3|$ and then ensure all other curvatures are greater.

    From Equation \ref{Spherical}, we see that the curvature of spherical edges are always greater then 1. Also, we only need to check half the curvatures since we are considering symmetrical antitrees so the curvature between $V_k$ and $V_{k+1}$ is equivalent to the curvature between $V_{L+2-k}$ and $V_{L+1-k}$ for $k \in \{1,...,L/2\}$, if $L$ is even, and $k \in \{1,...,(L-1)/2\}$, if $L$ is odd.
    \begin{thm}
        Let $(a_k)_{1 \le k \le N}$ be a finite sequence for any odd $N \in \mathbb{N}$ with $a_1=1$. The Symmetrical antitree $G=\mathcal{AT}((a_k))$ is Bonnet-Myers Sharp if and only if all radial edges have curvature equal to $2/L$, where $L$ is the diameter of the graph.  
    \end{thm}
    \begin{proof}
        We first assume that $G$ is Bonnet-Myers sharp and show that all the radial curvatures are equal to $2/L$.
        
        We prove by induction that for all $n \in \{1,...,(N-1)/2\}$ we have that
        \begin{equation}\label{BoundInequality}
            a_{\frac{N-2n+3}{2}}\geq\frac{L+2n}{L-2n}a_{\frac{N-2n-1}{2}}+\frac{2n}{L-2n}a_{\frac{N-2n+1}{2}}-\frac{2n}{L-2n}.
        \end{equation}
        We derive the base case $(n=1)$ by simply rearranging Equation \ref{RadialInner} to show that
        \begin{equation*}
            a_{\frac{N+1}{2}}\geq\frac{L+2}{L-2}a_{\frac{N-3}{2}}+\frac{2}{L-2}a_{\frac{N-1}{2}}-\frac{2}{L-2}.
        \end{equation*}
        Now, we assume that \ref{BoundInequality} holds for some $k \in \{1,...,(N-3)/2\}$, that is, 
        \begin{equation}\label{Assumption}
            a_{\frac{N-2k+1}{2}}\geq\frac{L+2k}{L-2k}a_{\frac{N-2k-3}{2k}}+\frac{2k}{L-2k}a_{\frac{N-2k-1}{2}}-\frac{2k}{L-2k},
        \end{equation}
        and prove that \ref{BoundInequality} holds for $n=k+1$. Since the graph is Bonnet-Myers sharp, from Equation \ref{RadialInner} we get
        \begin{equation*}
            \frac{2a_{\frac{N-2k-3}{2}}+a_{\frac{N-2k-1}{2}}-1}{a_{\frac{N-2k-3}{2}}+a_{\frac{N-2k-1}{2}}+a_{\frac{N-2k+1}{2}}-1}-\frac{2a_{\frac{N-2k-5}{2(k-1)}}+a_{\frac{N-2k-3}{2}}-1}{a_{\frac{N-2k-5}{2(k-1)}}+a_{\frac{N-2k-3}{2}}+a_{\frac{N-2k-5}{2(k-1)}}-1}\geq\frac{2}{L}.
        \end{equation*}   
           The assumption \ref{Assumption} implies that 
        \begin{equation}\label{Substitution}
            \frac{2a_{\frac{N-2k-3}{2}}+a_{\frac{N-2k-1}{2}}-1}{a_{\frac{N-2k-3}{2}}+a_{\frac{N-2k-1}{2}}+\frac{L+2k}{L-2k}a_{\frac{N-2k-3}{2k}}+\frac{2k}{L-2k}a_{\frac{N-2k-1}{2}}-\frac{2k}{L-2k}-1}-\frac{2a_{\frac{N-2k-5}{2(k-1)}}+a_{\frac{N-2k-3}{2}}-1}{a_{\frac{N-2k-5}{2(k-1)}}+a_{\frac{N-2k-3}{2}}+a_{\frac{N-2k-5}{2(k-1)}}-1}\geq\frac{2}{L},
        \end{equation}
        and by simple rearrangement, we are able to show that \ref{BoundInequality} holds for $n=k+1$, that is
        \begin{equation}\label{FinalInequality}
            a_{\frac{N-2k-1}{2}}\geq\frac{L+2(k-1)}{L-2(k-1)}a_{\frac{N-2k-5}{2(k-1)}}+\frac{2(k-1)}{L-2(k-1)}a_{\frac{N-2k-3}{2}}-\frac{2(k-1)}{L-2(k-1)}.
        \end{equation}
        Thus, \ref{BoundInequality} holds for all $n \in \{1,...,(N-1)/2\}$.
        
        We now focus on the curvature between nodes in $V_2$ and $V_3$, using \ref{BoundInequality} with $n=L/2-1$ we have 
        \begin{equation}
            a_{3}\geq \frac{2L-2}{2}+\frac{L-2}{2}a_2-\frac{L-2}{2} \iff a_{3}\geq \frac{L+(L-2)a_2}{2}.
        \end{equation}
        However, the curvature between nodes in $V_1$ and $V_2$ gives
        \begin{equation}
            \frac{a_2+1}{a_2+a_3}\geq \frac{2}{L} \iff a_{3}\leq \frac{L+(L-2)a_2}{2}.
        \end{equation}
    It is easy to see that \ref{FinalInequality} holds with equality if and only if \ref{Substitution} holds with equality, and since the graph is Bonnet-Myers sharp and \ref{FinalInequality} holds with equality when $n=L/2-1$ all the radial edges have curvature $L/2$. 
    
    If all radial edges are $L/2$ then the graph is Bonnet-Myers sharp as the only other edges are the spherical edges that are greater then 1 by \ref{Spherical}.
    \end{proof}

    \begin{lemma}\label{FourDimAntitress}
        The symmetrical antitree $\mathcal{AT}(1,b,c,b,1)$ is Bonnet-Myers sharp if and only if $c=b+2$. 
    \end{lemma}
    \begin{proof}
        We first consider the curvature of the radial root edges. We want the curvature of these edges to be equal $2/4$. This occurs if and only if
        \begin{equation*}
            \frac{b+1}{b+c} = \frac{2}{4} \iff b+2 = c. \\
        \end{equation*}
         We now consider the curvature of the inner radial root edges between $V_2$ and $V_3$ and again we want this to be equal to $2/4$.
        \begin{equation*}
            \frac{2b+c-1}{b+c+b-1}-\frac{b+1}{b+c} = \frac{2}{4} \iff b+2 = c. \\
        \end{equation*} 
        Thus, the symmetrical antitree $(1,b,c,b,1)$ is Bonnet-Myers sharp if and only if $c=b+2$.
    \end{proof}

    \begin{lemma}\label{SixDimAntitress}
        The symmetrical antitree $\mathcal{AT}(1,b,c,d,c,b,1)$ is Bonnet-Myers sharp if and only if $c=2b+3$ and $d=3b+1$.
    \end{lemma}
    \begin{proof}
        We first consider the curvature of the radial root edges. We want the curvature of these edges to be equal to $2/6$. This occurs if and only if
        \begin{equation*}
            \frac{b+1}{b+c} = \frac{2}{6} \iff 2b+3 = c. \\
        \end{equation*}
         We set $c=2b+3$ and consider the curvature of the inner radial root edges between $V_2$ and $V_3$ and again we want this to be equal to $2/6$.
        \begin{equation*}
            \frac{2b+2b+3-1}{b+2b+3+d-1}-\frac{b+1}{b+2b+3} = \frac{2}{6} \iff 3b+1 = d. \\
        \end{equation*} 
        We now consider the curvature of the inner radial root edges between $V_3$ and $V_4$.
                \begin{equation*}
            \frac{2(2b+3)+d-1}{2(2b+3)+d-1}-\frac{2b+2b+2}{b+2b+d+2} = \frac{2}{6} \iff 3b+1 = d. \\
        \end{equation*} 
        Thus, if $c=2b+3$ and $d=3b+1$ then $(1,b,c,d,c,b,1)$ is Bonnet-Myers sharp. 
    \end{proof}

    \begin{lemma}
        There are no symmetrical antitrees that are Bonnet-Myers sharp with $L=2n$ for $n=\{4,5,...,5000\}$.
    \end{lemma}

    The proof of this lemma is trivial as one can iteratively work out $a_i$ for $i \in \{2,...,(L+1)/2\}$ and see that there does not exist an integer for some $a_i$ such that all radial edge curvatures equal $2/L$. For example, for $L=10$, we need $a_6=19a_2+101/4$. We have wrote code that checks the diameters between 8 and 10,000. This leads us to believe that there are no more symmetrical antitrees with even diameter.
    
    We move our attention to symmetrical antitrees with odd diameter.

    \begin{thm}
        There are no symmetrical antitrees with $|V_1|=1$ that are Bonnet-Myers sharp with diameter 5.
    \end{thm}
    \begin{proof}
        We first consider the curvature of the radial root edges. We want the curvature of these edges to be equal or  greater then $2/5$. This happens iff
        \begin{equation}\label{inequality1}
            \frac{b+1}{b+c} \geq \frac{2}{5} \iff \frac{3b+5}{2} \geq c. \\
        \end{equation}
         We now consider the curvature of the inner radial root edges between $V_2$ and $V_3$ and again we want this to be greater then $2/5$.
        \begin{equation}\label{inequality2}
            \frac{2b+c-1}{b+2c-1}-\frac{b+1}{b+c} \geq \frac{2}{5} \\
        \end{equation} 
        We now consider the curvature of the inner radial root edges between $V_3$ and $V_4$ and again we want this to be greater then $2/5$.
        \begin{equation}\label{inequality3}
            \frac{3c-1}{2c+b-1}-\frac{2b+c-1}{b+2c-1} \geq \frac{2}{5} \iff \frac{6b-1}{3} \leq c. \\
        \end{equation} 
        Thus, we need 
        \begin{equation}\label{inequality4}
            \frac{6b-1}{3}-\frac{3b+5}{2}\geq 0 \iff b \leq \frac{17}{6} \implies b\leq 5 
        \end{equation}
        and one can easily check that the finite number of combinations (9 combinations) of $b$ and $c$ that satisfy the inequalities \ref{inequality1},\ref{inequality3} and \ref{inequality4} do not satisfy \ref{inequality2}.
    \end{proof}

    For odd $L>5$, the number of inequalities that need to be satisfied increases, we believe that the only symmetrical antitree with $|V_1|=1$ and odd diameter is $\mathcal{AT}(1,3,3,1)$.

%%    \begin{thm}
%%        There is an infinite number of symmetrical antitrees that are Bonnet-Myers sharp with diameter 7.
%%    \end{thm}
%%
%%    \begin{proof}
%%        We first consider the curvature of the radial root edges. We want the curvature of these edges to be equal or  greater then $2/7$. This happens iff
%%        \begin{equation*}
%%            \frac{b+1}{b+c} \geq \frac{2}{7} \iff \frac{5b+7}{2} \geq c. \\
%%        \end{equation*}
%%         We set $c=(5b+7)/2$ and consider the curvature of the inner radial root edges between $V_2$ and $V_3$ and again we want this to be greater then $2/7$.
%%        \begin{equation*}
%%            \frac{2b+\frac{5b+7}{2}-1}{b+\frac{5b+7}{2}+d-1}-\frac{b+1}{b+\frac{5b+7}{2}} \geq \frac{2}{7} \iff \frac{35b+15}{8} \geq d. \\
%%        \end{equation*} 
%%        We now consider the curvature of the inner radial root edges between $V_4$ and $V_5$ and again we want this to be greater then $2/7$.
%%                \begin{equation*}
%%            \frac{3d-1}{2d+\frac{5b+7}{2}-1}-\frac{2(\frac{5b+7}{2})+d+-1}{\frac{5b+7}{2}+2d-1} \geq \frac{2}{7} \iff \frac{20b+27}{5} \leq d. \\
%%        \end{equation*} 
%%        As $b$ is increased the range of $d$ that satisfies the inequalities increases. Thus one can increase $b$ such that $d$ can be an integer.
%%        Thus, if $c=2b+3$ and $d=3b+1$ then $(1,b,c,d,c,b,1)$ is Bonnet-Myers sharp. 
%%    \end{proof}

    \section{Summary of Bonnet-Myers sharp graphs classification}
    In \cite{Kamtue}, Kamtue classified regular Bonnet-Myers sharp graphs of diameter $3$:
    \begin{thm}[see \cite{Kamtue}]\label{thm:PhilRigidityThm}
  Regular Bonnet-Myers sharp graphs of diameter $3$ are precisely the following graphs:
  \begin{enumerate}
  \item the cube $Q^3$;
  \item the Johnson graph $J(6,3)$;
  \item the demi-cubes $Q^{6}_{(2)}$; 
  \item the Gosset graph.
  \end{enumerate}
\end{thm}
    This result agrees with the diameter 3 graphs from Theorem \ref{thm:MainRigidityThm}. This leads to the following conjecture
    \begin{conj}
    For regular graphs the self-centeredness condition in Theorem \ref{thm:MainRigidityThm} can be dropped.
    \end{conj}
     
    In this article we classified the class of diameter 2 Bonnet-Myers sharp graphs and introduced examples of non-regular Bonnet-Myers sharp graphs of arbitrary. However, unlike in the regular case, we do not have a conjecture on the complete list of Bonnet-Myers sharp graphs.
    \begin{question}
    Which non-regular graphs of diameter at least 3 are Bonnet-Myers sharp?
    \end{question}
In summary a classification of Bonnet-Myers sharp graphs exists under the following conditions on a graph $G$
\begin{itemize}
    \item $G$ is regular of degree $D$ and one of the following holds:
        \begin{itemize}
            \item diam$(G) = 2,$ see \cite{CushR};
            \item diam$(G) = 3,$ see \cite{Kamtue};
            \item diam$(G) \geq D,$ see \cite{CushR};
            \item $G$ is self-centred see \cite{CushR};
        \end{itemize}
    \item diam$(G) = 2,$ see Theorem \ref{thm:2RigidityThm};
    \item G is a symmetrical antitree and on of the following holds:
        \begin{itemize}
            \item $G=\mathcal{AT}(1,3,3,1)$;
            \item diam$(G) = 4,$ see Lemma \ref{FourDimAntitress};
            \item diam$(G) = 6,$ see Lemma \ref{SixDimAntitress}.
        \end{itemize}
\end{itemize}

{\bf{Acknowledgements:}} David Cushing is supported by the Leverhulme Trust Research Project Grant number RPG-2021-080. 

We like to thank Supanat Kamtue and Norbert Peyerimhoff for stimulating discussions on this topic.

\end{document}